\documentclass[11pt,twoside]{amsart}
\usepackage{latexsym,amssymb,amsmath}
\usepackage[all]{xy}

\numberwithin{equation}{section}
\newtheorem{theorem}{Theorem}[section]
\newtheorem{lemma}[theorem]{Lemma}
\newtheorem{proposition}[theorem]{Proposition}
\newtheorem{corollary}[theorem]{Corollary}

\theoremstyle{definition}
\newtheorem{definition}[theorem]{Definition}
\newtheorem{procedure}[theorem]{Procedure}

\newtheorem{example}[theorem]{Example}

\begin{document}

\title[Weighted Ehrhart functions]{Weighted Ehrhart functions}

\thanks{The authors were supported by SNII, M\'exico.}

\author[E. Reyes]{Enrique Reyes}
\address{
Departamento de Matem\'aticas\\
Cinvestav, Av. IPN 2508, 07360, CDMX, M\'exico.
}
\email{ereyes@math.cinvestav.mx}

\author[C. E. Valencia]{Carlos E. Valencia}
\address{
Departamento de Matem\'aticas\\
Cinvestav, Av. IPN 2508, 07360, CDMX, M\'exico.
}
\email{cvalencia@math.cinvestav.edu.mx}
\author[R. H. Villarreal]{Rafael H. Villarreal}
\address{
Departamento de Matem\'aticas\\
Cinvestav, Av. IPN 2508, 07360, CDMX, M\'exico.
}
\email{rvillarreal@cinvestav.mx}


\dedicatory{Festschrift on the occasion of the 80th anniversary of the Mexican Mathematical Society}  

\begin{abstract} 
We give an algorithm for computing weighted Ehrhart functions of
lattice polytopes with polynomial weights on its lattice points using Lagrange interpolation.  
We show how to compute generating functions of polynomials using
those of unit cubes and Eulerian numbers, and apply integer
programming to study the algebraic properties of the Ehrhart ring of
the $d$-th unit cube.    
We then present some applications to weighted Ehrhart functions and
enumeration problems.  
\end{abstract}

\maketitle

\section{Introduction}

Let $S=K[t_1,\ldots,t_s] $ be a polynomial ring over a field $K$ and
let $\mathbb{N}=\{0,1,\ldots\}$ be the non-negative integers.  
A function $E\colon\mathbb{N}\rightarrow K$ is called a
\textit{polynomial} in $n$ if there is a polynomial $g(x)\in K[x]$
such that $E(n)=g(n)$ for all $n\in\mathbb{N}$.   
When $E(n)=g(n)$ for all $n\gg 0$, we say that $E$ is a \textit{polynomial function}.

Let $\mathcal{P}={\rm conv}(v_1,\ldots,v_m)$ be a lattice polytope in
$\mathbb{R}^s$ of dimension $d$ whose vertices are in $\mathbb{N}^s$.  
The \textit{Ehrhart function}
$E_\mathcal{P}\colon\mathbb{N}\rightarrow\mathbb{N}$, $n\mapsto
|n\mathcal{P}\cap\mathbb{Z}^s|$, of $\mathcal{P}$ is a polynomial in
$n$ of degree $d$ whose leading coefficient is the relative volume
${\rm vol}(\mathcal{P})$ of $\mathcal{P}$ (see, for instance,
\cite{BeckRobins}), and the generating function of $F_\mathcal{P}$ of
$E_\mathcal{P}$ is a rational function of the form      
$$
F_\mathcal{P}(x):=\sum_{n=0}^\infty E_\mathcal{P}(n)x^n=\frac{h(x)}{(1-x)^{d+1}}.
$$
\quad By the famous positivity theorem of Stanley in Ehrhart theory
\cite[Theorem~2.1]{Stanley-nonneg-h-vector}, $h(x)$ is a polynomial
with nonnegative integer coefficients of degree at most $d$.   
On the ring theory side, the \textit{Ehrhart ring} of $\mathcal{P}$,
denoted $A(\mathcal{P})$, is the monomial subring of $S[z]$ given by
\begin{equation*}
A(\mathcal{P}):=K[\{t^az^n\mid a\in n\mathcal{P}\cap\mathbb{Z}^s\}],
\end{equation*}
where $z$ is a \emph{grading} variable. 
The Ehrhart ring $A(\mathcal{P})$ has a natural $\mathbb{N}$-grading given by  
$$
A(\mathcal{P})=\bigoplus_{n=0}^\infty A(\mathcal{P})_n,
$$
where $t^az^n\in A(\mathcal{P})_n$ if and only if $a\in
n\mathcal{P}\cap\mathbb{Z}^s$.  
The Ehrhart ring $A(\mathcal{P})$ is normal and it is a finitely
generated $K$-algebra, that is,
$A(\mathcal{P})=K[\{t^{c_i}z^{d_i}\mid i=1,\ldots,r\}]$.  
See \cite[Theorem~9.3.6]{monalg-rev} for a summary of properties of Ehrhart rings.  
  
Ehrhart's theory has been the subject of considerable attention due
to its applications in combinatorics and commutative algebra. 
Brion and Vergne \cite{Brion-Vergne} presented in 1997 a
generalization of Ehrhart's theorem where the lattice points are counted with
a ``weight'' given by evaluating a polynomial $w$ with coefficients in a field
$K$ of characteristic zero, i.e.,    
$$
E_\mathcal{P}^w(n):=\sum_{\alpha \in n\mathcal{P} \cap \mathbb{Z}^s} w(\alpha).
$$  
\quad The function $E_\mathcal{P}^w$ is the \textit{weighted Ehrhart
function} of $\mathcal{P}$  and its generating function  
$$
F_\mathcal{P}^w(x):=\sum_{n=0}^\infty E_\mathcal{P}^w(n)x^n
$$
is the \textit{weighted Ehrhart series} of $\mathcal{P}$. 
Weighted Ehrhart theory has been developed in several papers,
see~\cite{jdl3,baldoni-etal-1,baldoni-etal,Brion-Vergne,bruns-ichim-soger,Bruns-Soger}
and references therein.   
We thank Jes\'us De Loera for introducing us to this theory while he
was on sabbatical at Cinvestav in the fall of 2022. The recent
paper~\cite{ehrhartw} studies Ehrhart functions and series of
weighted lattice points.    

In Section~\ref{section-weighted-ehrhart}, we recall some facts from
weighted Ehrhart theory and show some applications to weighted
Ehrhart functions.   

 
The Lagrange interpolating polynomial is the unique univariate
polynomial $f(x)$ of degree at most $e$ that passes through $e+1$
distinct points in $\mathbb{R}^2$.    
In Appendix~\ref{procedures-wehrhart}, we give an algorithm
implemented in \textit{Macaulay}$2$ \cite{mac2}---using its interface
to \textit{Normaliz} \cite{normaliz2}---to compute the weighted
Ehrhart polynomial and the weighted Ehrhart series of a lattice
polytope $\mathcal{P}$ relative to a polynomial $w$.     
This algorithm is based on the fact that $E_\mathcal{P}^w$ is a
polynomial in $n$ of degree at most $\dim(\mathcal{P})+\deg(w)$
(Theorem~\ref{nonneg-poly}(a)), which allows us to use Lagrange
interpolation \cite{interpolation} to compute $E_\mathcal{P}^w$.    
Then, using the fact that the generating function of a polynomial in
$n$ can be computed using the unit cubes $C_d:=[0,1]^d$, $d\geq 0$
(Lemma~\ref{oct12-22}), we can compute the weighted Ehrhart series
$F_\mathcal{P}^w$.    
The algorithm is implemented and explained in
Procedure~\ref{procedure0}.  
In Section~\ref{examples-section}, we provide several examples that
can be computed using this procedure. 

One can also use \textit{Normaliz} \cite{normaliz2,Bruns-Soger}
directly (Example~\ref{jun4-25-1}) or finite differences methods to
compute $E_\mathcal{P}^w$ and $F_\mathcal{P}^w$ \cite{Sauer,Sauer-Xu}.   
If $w$ is linear, one can give a procedure for \textit{Normaliz}
\cite{normaliz2} based on Proposition~\ref{linear-weighted-ehrhart} to compute $E_\mathcal{P}^w$ and $F_\mathcal{P}^w$.   

In Section~\ref{ehrhart-cube}, we explain how to compute generating functions of polynomials using those of unit cubes and
Eulerian numbers and, motivated by this, we use integer programming to study the algebraic properties of the Ehrhart ring of the $d$-th unit cube.

We consider the following enumeration problem. 
Let $w_i\colon\mathbb{R}^s\rightarrow\mathbb{R}$, $i=1,\ldots,p$, be linear functions such that $w_i(e_j)\in\mathbb{N}$ for all $i,j$. 
Is the function $H\colon\mathbb{N}\rightarrow\mathbb{N}$,
$$
H(n):=|w(n\mathcal{P}\cap\mathbb{Z}^s)|,\ w=(w_1,\ldots,w_p),
$$
a polynomial function? 
In Section~\ref{enumeration-problem}, we give a positive answer to this problem and give a method to compute $H$ by showing that the following \textit{weighted Ehrhart ring} $A^w(\mathcal{P})$ is a finitely generated graded $K$-algebra, generated by monomials, whose Hilbert function is $H$:
$$
A^w(\mathcal{P}):=K[\{y^{w(a)}z^n\mid a\in n\mathcal{P}\cap\mathbb{Z}^s\}]\subset K[y_1,\ldots,y_p,z],
$$
where $y^{w(a)}:=y^{(w_1(a),\ldots,w_p(a))}=y_1^{w_1(a)}\cdots y_p^{w_p(a)}$ for $a\in\mathbb{N}^s$ and $K[y_1,\ldots,y_p,z]$ is a polynomial ring over the field $K$, see Theorem~\ref{jdl2}. 
Several types of weighted Ehrhart rings and weighted Ehrhart functions are studied by De Loera et al. in \cite{ehrhartw}. 

For unexplained terminology and additional information, we refer to \cite{BeckRobins,BG-book} for Ehrhart theory and \cite{BHer,Sta1,Sta5,Sta2} for commutative algebra and enumerative combinatorics.

\section{Weighted Ehrhart theory}\label{section-weighted-ehrhart}
 
\begin{definition}\rm 
We say that a lattice polytope $\mathcal{P}={\rm conv}(v_1,\ldots,v_m)$ is \textit{non-degenerate} if for each $1\leq i\leq s$, there is $v_j$ such that the $i$-th entry of $v_j$ is non-zero. 
\end{definition}


\begin{theorem}\label{nonneg-poly}
If $w$ is a non-zero polynomial of $S=K[t_1,\ldots,t_s]$ of degree
$p$, $\mathcal{P}$ is a lattice polytope of dimension $d$, and
$\mathbb{R}\subset K$, then the following holds\/: 
\begin{enumerate}
\item[\rm(a)] {\rm\cite{ehrhartw}} $E_\mathcal{P}^{w}$ is a
polynomial in $n$ of degree at most $d+p$ and $F_\mathcal{P}^{w}(x)$
is a rational function.    
\item[\rm(b)] {\rm\cite{ehrhartw}} If $\mathcal{P}$ is non-degenerate
and $w$ is a monomial, then $\deg(E_\mathcal{P}^{w})=d+p$.    
\item[\rm(c)] {\rm\cite{ehrhartw}} $E_\mathcal{P}^{w}$ is a
$K$-linear combination of Ehrhart polynomials.   
\item[\rm(d)]{\rm\cite[Proposition~4.1]{Brion-Vergne}} If the
interior $\mathcal{P}^{\rm o}$ of $\mathcal{P}$ is nonempty,
$K=\mathbb{R}$, and $w$ is homogeneous with $w\geq 0$ on $\mathcal{P}$, then $E_\mathcal{P}^{w}$ is a polynomial of degree $s+p$.        
\end{enumerate}
\end{theorem}

In Theorem~\ref{nonneg-poly}(d) the leading coefficient of $E_\mathcal{P}^w$ is equal to $\int_\mathcal{P}w$. 
This fact appears in \cite[p.~437]{baldoni-etal} and \cite[Proposition~5]{Bruns-Soger}. 
Integrals of the type $\int_\mathcal{P}w$ with $w$ a polynomial and $\mathcal{P}$ a rational polytope were studied in \cite{baldoni-etal-1,barvinok,barvinok1,Bruns-Soger,lasserre}.
Algorithms and the implementations to compute this integral were developed independently in \textit{LattE integrale} \cite{latte-integrale} and \textit{Normaliz} \cite{normaliz2} (see Example~\ref{jun4-25-1}).  

\begin{corollary}\label{jan8-23-1} 
Let $G$ be a graph with vertex set $V(G)=\{t_1,\ldots,t_s\}$ and without isolated vertices, let $\mathcal{P}$ be the convex hull of all $e_i+e_j$ in $\mathbb{R}^s$ such that $\{t_i,t_j\}$ is an edge of $G$, and let $c_0$ be the number of bipartite connected components of $G$. 
If $w$ is a monomial in $S$, then 
$$
\deg(E_\mathcal{P}^w)=s-c_0-1+\deg(w).
$$
\end{corollary}

\begin{proof} As $G$ has no isolated vertices, $\mathcal{P}$ is non-degenerate, and the result follows readily from Corollary~\ref{nonneg-poly}(b) and from the fact that $\dim(\mathcal{P})=s-c_0-1$ \cite[Proposition~10.4.1]{monalg-rev}.
\end{proof}

\begin{proposition}{\rm\cite{ehrhartw}}\label{linear-weighted-ehrhart}
Let $\mathcal{P}={\rm conv}(v_1,\ldots,v_m)$ be a lattice polytope
whose vertices are in $\mathbb{N}^s$ and let
$w\colon\mathbb{R}^s\rightarrow\mathbb{R}$ be a non-zero linear
function such that $w(e_i)\in\mathbb{N}$ for $i=1,\ldots,s$.   
Then,      
$$
E_\mathcal{P}^{w}(n):=\sum_{a\in
n\mathcal{P}\cap\mathbb{Z}^s}w(a)=E_{\mathcal{P}_{w}}(n)-E_\mathcal{P}(n)\mbox{
for all }n\geq 0,    
$$
where ${\mathcal{P}_w}={\rm conv}((v_1,0),\ldots,(v_m,0),(v_1,w(v_1)),\ldots(v_m,w(v_m)))$.
\end{proposition}

\begin{corollary}\label{affine-weighted-ehrhart} 
Let $\mathcal{P}={\rm conv}(v_1,\ldots,v_m)$ be a lattice polytope whose vertices are in $\mathbb{N}^s$.
If $\omega\colon\mathbb{R}^s\rightarrow\mathbb{R}$ is an affine
function of the form $w(x)=Cx+b$, $C(c_1,\ldots,c_s)\in\mathbb{R}^s$,
$b\in\mathbb{R}$, such that $Ce_i\in\mathbb{N}$ for $i=1,\ldots,s$,
then        
$$
E_\mathcal{P}^w(n):=\sum_{a\in
n\mathcal{P}\cap\mathbb{Z}^s}w(a)=E_{\mathcal{Q}_1}(n)+(b-1)E_\mathcal{P}(n)\mbox{
for all }n\geq 0,  
$$
where $\mathcal{Q}_1={\rm conv}((v_1,0),\ldots,(v_m,0),(v_1,Cv_1),\ldots(v_m,Cv_m))$.
\end{corollary}

\begin{proof} Letting $w_1(x)=Cx$ and applying Proposition~\ref{linear-weighted-ehrhart}, we obtain
\begin{align*}
E_\mathcal{P}^{w_1}(n)&=\sum_{a\in
n\mathcal{P}\cap\mathbb{Z}^s}w_1(a)=\sum_{a\in
n\mathcal{P}\cap\mathbb{Z}^s}(w(x)-b)=E_\mathcal{P}^{w}(n)-bE_\mathcal{P}(n)=E_{\mathcal{Q}_1}(n)-E_\mathcal{P}(n)
 \end{align*}
for all $n\geq 0$, and consequently
$E_\mathcal{P}^{w}(n)=E_{\mathcal{Q}_1}(n)+(b-1)E_\mathcal{P}(n)$ for
all $n\geq 0$. 
\end{proof}

\begin{proposition}\label{jan8-23} 
Let $\mathcal{P}$ be a lattice polytope in $\mathbb{R}^s$ of dimension $s$ and let $f\neq 0$ be a homogeneous polynomial of $S$. 
If $f\geq 0$ on $\mathcal{P}$ and $r$ is a negative integer root of $E_\mathcal{P}$, then $E_\mathcal{P}^f(r)=0$. 
\end{proposition}

\begin{proof} Let $p$ be the degree of $f$ and let $\mathcal{P}^{\rm o}$ be the interior of $\mathcal{P}$. 
By the generalized reciprocity law of Ehrhart \cite[Proposition~4.1]{Brion-Vergne}, for $n\geq 1$ one has
\begin{align*} 
&E_{\mathcal{P}^{\rm o}}(n)=(-1)^{s}E_\mathcal{P}(-n),\\
&E_{\mathcal{P}^{\rm o}}^f(n)=(-1)^{s+p}E_\mathcal{P}^f(-n).
\end{align*}
\quad Setting $n=-r$, from the first equality we get $E_{\mathcal{P}^{\rm o}}(n)=0$, that is, $n\mathcal{P}^{\rm o}\cap\mathbb{Z}^s=\emptyset$. 
Then, $E_{\mathcal{P}^{\rm o}}^f(n)=0$ and, from the second equality, we obtain $E_\mathcal{P}^f(-n)=0$.
\end{proof}

\section{The Ehrhart ring of the $d$-th unit cube}\label{ehrhart-cube}


In this section, we show how to compute generating functions of polynomials using those of unit cubes and Eulerian numbers, and use integer programming to study the algebraic properties of the Ehrhart ring of the $d$-th unit cube. 

\subsection*{Generating functions of polynomials}

For natural numbers $0\leq k\leq d$, the \textit{Eulerian number} $A(d,k)$ is defined by  
\begin{equation}\label{oct9-1-22}
\sum_{n=0}^\infty n^dx^n=\frac{\sum_{k=0}^dA(d,k)x^k}{(1-x)^{d+1}},
\end{equation}
where by convention $A(d,k)=0$ whenever $k>d$. 
By \cite[Corollary~4.3.1]{Sta5}, Eulerian numbers are well defined and by \cite[Proposition~A.4(d)]{Duran}, for $d\in\mathbb{N}_+$ one has
\begin{equation}\label{oct9-2-22}
A(d,k)=\sum_{j=0}^k(-1)^j\binom{d+1}{j}(k-j)^d,\quad 0\leq k\leq d.
\end{equation}
\quad On the other hand, let $C_d=[0,1]^d={\rm conv}(\{0,1\}^d)$ be
the $d$-th unit cube, $d\geq 1$, and let $C_0={\rm conv}(\{1\})$. It
is well known that the Ehrhart polynomial of $C_d$ is $(n+1)^d$ and
that the Ehrhart series of $C_d$ is
\begin{equation}\label{oct9-3-22}
F_{C_d}(x)=\frac{\sum_{k=1}^dA(d,k)x^{k-1}}{(1-x)^{d+1}}
\end{equation}
if $d\geq 1$ and $F_{C_0}(x)=1/(1-x)$ if $d=0$, see \cite[p.~31]{Duran}.

\begin{proposition}{\rm\cite[Corollary~4.3.1]{Sta5}} 
The generating function $\sum_{n=0}^\infty g(n)x^n$ of a polynomial $g(n)=\sum_{i=0}^r b_in^i$, $n\in\mathbb{N}$, $b_i\in K$, is a rational function. 
\end{proposition}

Unit cubes are useful to determine the generating function of $g$ because one has:  

\begin{lemma}\label{oct12-22} If $g(n)=b_0+b_1n+\cdots+b_rn^r$ for
$n\in\mathbb{N}$, $b_i\in K$, then 
$$
\sum_{n=0}^{\infty}g(n)x^n=b_0F_{C_0}(x)+\sum_{i=1}^rb_ixF_{C_i}(x). 
$$
\end{lemma}
\begin{proof}  
By Eq.~\eqref{oct9-2-22}, $A(d,0)=0$. 
Then, from Eqs.~\eqref{oct9-1-22} and \eqref{oct9-3-22}, we obtain that 
\begin{align*}
\sum_{n=0}^\infty n^dx^n=\frac{\sum_{k=0}^dA(d,k)x^k}{(1-x)^{d+1}}
=\frac{\sum_{k=1}^dA(d,k)x^k}{(1-x)^{d+1}}
=\frac{x\sum_{k=1}^dA(d,k)x^{k-1}}{(1-x)^{d+1}}=xF_{C_d}(x)
\end{align*}
for $d\geq 1$, and consequently $\sum_{n=0}^\infty n^dx^n=xF_{C_d}$ for $d\geq 1$. 
Hence
\begin{align*}
\sum_{n=0}^{\infty}g(n)x^n&=b_0\sum_{n=0}^\infty x^n+b_1\sum_{n=0}^\infty nx^n+b_2\sum_{n=0}^\infty n^2x^n+\cdots+b_r\sum_{n=0}^\infty n^rx^n\\ 
&=b_0F_{C_0}(x)+\sum_{i=1}^rb_ixF_{C_i}(x),
\end{align*}
and the proof is complete.
\end{proof} 

Using the explicit formulas to compute $F_{C_d}$ given in Eqs.~\eqref{oct9-2-22}-\eqref{oct9-3-22}, we can compute weighted Ehrhart series by recalling that weighted Ehrhart functions are polynomials \cite{Brion-Vergne} and that they can be computed using polynomial interpolation (Theorem~\ref{nonneg-poly}(a)). 

\subsection*{The Ehrhart ring of the $\mathbf{d}$-th unit cube}

\begin{proposition}\label{ehrhart-cube-prop} 
Let $\mathcal{P}=[0,1]^d$ be the $d$-th unit cube in $\mathbb{R}^d$ and let $A(\mathcal{P})$ be the Ehrhart ring of
$\mathcal{P}$. 
The following hold
\begin{enumerate} 
\item[\rm(a)] The polytopal ring $K[\mathcal{P}]:=K[t^az\mid a\in\mathcal{P}\cap\mathbb{Z}^d]$ is normal.
\item[\rm(b)] $A(\mathcal{P})=K[\mathcal{P}]$.
\item[\rm(c)] $A(\mathcal{P})$ is a Gorenstein ring whose $a$-invariant is equal to $-2$.
\end{enumerate}
\end{proposition}
\begin{proof}  
(a) Let $D$ be the $d\times 1$ matrix $(1,\ldots,1)^\top$. 
Note that $\mathcal{P}\cap\mathbb{Z^d}=\{0,1\}^d$, $\mathcal{P}$ is equal to ${\rm conv}(\{0,1\}^d)$, and the polytopal ring of $\mathcal{P}$ can be written as
\begin{align*}
K[\mathcal{P}]&=K[t^az\mid a=0\mbox{ or }a=e_{i_1}+\cdots+e_{i_k},\, 1\leq i_1<\cdots<i_k\leq d]\\
&=K[t^az\mid a\in\mathbb{N}^d,\, a\leq (1,\ldots,1)].
\end{align*}
\quad For each vector $a=(a_1,\ldots,a_d)\in\mathbb{N}^d$, one has 
\begin{align*}
\lceil{\rm min}\{\langle y,{1}\rangle\mid y\geq 0;\, Dy\geq a \}\rceil&=\lceil{\rm min}\{y\mid y\geq 0;\, y\geq a_i,\,
i=1,\ldots,d\}\rceil =\max\{a_i\}_{i=1}^d,\\
{\rm min}\{\langle y,{1}\rangle\mid y\geq 0;\, Dy\geq a \}&={\rm min}\{y\mid y\in\mathbb{N};\, y\geq a_i,\, i=1,\ldots,d\}=\max\{a_i\}_{i=1}^d.
\end{align*}
\quad This means that the linear system $x\geq 0; xD\leq{1}$ has the \textit{integer rounding property} in the sense of
\cite[Definition~2.2, Remark~2.3]{ainv}. 
Hence, by \cite[Theorem~2.5]{ainv}, the ring $K[\mathcal{P}]$ is normal. 
The normality of $K[\mathcal{P}]$ also follows using the fact that the $d$-th unit cube $[0,1]^d$ has a unimodular triangulation \cite{BGT,deloera-book}. 

(b) Consider the subset of $\mathbb{R}^{d+1}$ given 
$$
\mathcal{U}:=\{(0,1)\}\cup\{(e_{i_1}+\cdots+e_{i_k},1)\mid 1\leq
i_1<\cdots<i_k\leq d\},
$$
and let $C$ be the matrix whose columns are the vectors in $\mathcal{U}$. 
The greatest common divisor of all the non-zero $(d+1)\times(d+1)$ minors of $C$ is equal to $1$ because $C$ diagonalizes over $\mathbb{Z}$ to an ``identity matrix'' $[I_{d+1},0]$. 
Then, by \cite[Theorem~3.9]{ehrhart}, the integral closure of $K[\mathcal{P}]$ is equal to $A(\mathcal{P})$. 
Thus, $K[\mathcal{P}]=A(\mathcal{P})$ because $K[\mathcal{P}]$ is normal by part (a).

(c) The set of vertices of the polytope $\{x\mid x\geq 0; xD\leq{1}\}$ is equal to $\{0,e_1,\ldots,e_d\}$ and, by the proof of part (a), the linear system $x\geq 0; xD\leq{1}$ has the integer rounding property.
Then, by \cite[Theorem~4.2]{ainv}, the $a$-invariant of the ring $K[\mathcal{P}]$ is $-\max\{1+|e_i|\}_{i=1}^d=-2$ and, by \cite[Theorem~4.4]{ainv}, $K[\mathcal{P}]$ is a Gorenstein ring. 
The canonical module $\omega_{K[\mathcal{P}]}$ of $K[\mathcal{P}]$ is generated by $f=t_1\cdots t_sz^2$, that is, $\omega_{K[\mathcal{P}]}=fK[\mathcal{P}]$.
\end{proof}


\section{Interaction of weighted Ehrhart theory and enumeration problems}\label{enumeration-problem}

Let $\mathcal{P}={\rm conv}(v_1,\ldots,v_m)$ be a lattice polytope of dimension $d$ whose vertices are in $\mathbb{N}^s$.
Let $w_i\colon\mathbb{R}^s\rightarrow\mathbb{R}$, $i=1,\ldots,p$, be linear functions such that $w_i(e_j)\in\mathbb{N}$ for all $i,j$ and let $K[y_1,\ldots,y_p,z]$ be a polynomial ring over the field $K$. 
We let $w:=(w_1,\ldots,w_p)$ and 
$$
y^{w(a)}:=y^{(w_1(a),\ldots,w_p(a))}:=y_1^{w_1(a)}\cdots y_p^{w_p(a)} \ \mbox{ for }a\in\mathbb{N}^s.
$$
\quad We are interested in the following problem:
Is the function $H\colon\mathbb{N}\rightarrow\mathbb{N}$,
$$
H(n):=|w(n\mathcal{P}\cap\mathbb{Z}^s)|,
$$
a polynomial function? 
We give a positive answer to this problem and provide a method to compute $H$ by identifying a weighted Ehrhart ring $A^w(\mathcal{P})$ which is a finitely generated graded $K$-algebra, generated by monomials, whose Hilbert function is $H$:
$$
A^w(\mathcal{P}):=K[\{y^{w(a)}z^n\mid a\in n\mathcal{P}\cap\mathbb{Z}^s\}]\subset K[y_1,\ldots,y_p,z],
$$
where $z$ is a new variable, see Theorem~\ref{jdl2}. 
The computation of $H$ reduces to the computation of the Hilbert function of the toric ideal of $A^w(\mathcal{P})$.  
Part of the difficulty in solving this problem is that despite the equality $w(n\mathcal{P})=nw(\mathcal{P})$, the inclusion 
$$
w(n\mathcal{P}\cap\mathbb{Z}^s)\subset w(n\mathcal{P})\cap\mathbb{Z}^p
$$
can be strict (Example~\ref{example7}). 
The hypothesis that $w_i$ is linear for $i=1,\ldots,p$ is essential to prove that $A^w(\mathcal{P})$ is finitely generated (see \cite[Proposition 2.2]{ehrhartw}).

\begin{theorem}\label{jdl2}
Let $\mathcal{P}={\rm conv}(v_1,\ldots,v_m)$ be a lattice polytope
whose vertices are in $\mathbb{N}^s$ and let $B$ be the matrix whose
columns are $\{(w(v_i),1)\}_{i=1}^m$.  
Let $w_i\colon\mathbb{R}^s\rightarrow\mathbb{R}$, $i=1,\ldots,p$, be
linear functions such that $w_i(e_j)\in\mathbb{N}$ for all $i,j$ and
let $w=(w_1,\ldots,w_p)$.   
The following holds.
\begin{enumerate}
\item[\rm (a)] $A^w(\mathcal{P}):=K[\{y^{w(a)}z^n\mid a\in
n\mathcal{P}\cap\mathbb{Z}^s\}]=\bigoplus_{i=0}^{\infty}A^w(\mathcal{P})_n$
is a graded $K$-algebra whose $n$-th graded component is given by   
\begin{align*}
A^w(\mathcal{P})_n:=\sum_{\scriptstyle a\in n\mathcal{P}\cap\mathbb{Z}^{s}}\hspace{-4mm}K y^{w(a)}z^n,\ \forall\, n\geq 0.
\end{align*}
\item[\rm (b)] The Hilbert function $H_{A^w(\mathcal{P})}$ of $A^w(\mathcal{P})$ is given by  
$$
H_{A^w(\mathcal{P})}(n):=\dim_K(A^w(\mathcal{P})_n)=|w(n\mathcal{P}\cap\mathbb{Z}^{s})|,\,\ \forall\, n\geq 0.
$$ 
\item[\rm (c)] There are $t^{c_1}z^{d_1},\ldots,t^{c_r}z^{d_r}\in A(\mathcal{P})$ such that 
$$
A(\mathcal{P})=K[\{t^{c_i}z^{d_i}\mid i=1,\ldots,r\}]\ \mbox{ and }\
A^w(\mathcal{P})=K[\{y^{w(c_i)}z^{d_i}\mid i=1,\ldots,r\}].  
$$
\item[\rm (d)] $L:=K[\{y^{w(v_i)}z\mid i=1,\ldots,m\}]\subset
A^w(\mathcal{P})$ is a finite integral extension. 
\item[\rm (e)] $\overline{L}\subset\overline{A^w(\mathcal{P})}\subset
A(w(\mathcal{P}))$ with equality everywhere if and only if the gcd of
all non-zero $k\times k$ minors of the matrix $B$ is equal to $1$, where $k={\rm rank}(B)$.  
\item[\rm (e)] The Hilbert function $H_{A^w(\mathcal{P})}$ of
$A^w(\mathcal{P})$ is a polynomial function of degree
$\dim(A^w(\mathcal{P}))-1$ and the Hilbert series
$F_{A^w(\mathcal{P})}$ of $A^w(\mathcal{P})$ has the form   
$$
F_{A^w(\mathcal{P})}(x)=\frac{h(x)}{(1-x)^{\dim(A^w(\mathcal{P}))}},
$$ 
where $h(1)\neq 0$ and $h(x)\in\mathbb{Z}[x]$.
\end{enumerate}
\end{theorem}
\begin{proof} 
(a) The inclusion $A^w(\mathcal{P})_kA^w(\mathcal{P})_n\subset A^w(\mathcal{P})_{k+n}$ for all $k,n\in\mathbb{N}$ follows from the linearity of $w_1,\ldots,w_p$ and the convexity of $\mathcal{P}$. 
Indeed, let $y^{w(a)}z^k\in A^w(\mathcal{P})_k$ and $y^{w(b)}z^n\in A^w(\mathcal{P})_n$. 
Then, $a\in k\mathcal{P}\cap\mathbb{Z}^s$, $b\in n\mathcal{P}\cap\mathbb{Z}^s$, and by the convexity of $\mathcal{P}$, one has that $(a+b)/(k+n)\in\mathcal{P}$. 
Hence, by the linearity of $w$, we get $(y^{w(a)}z^k)(y^{w(b)}z^n)=y^{w(a+b)}z^{k+n}\in A^w(\mathcal{P})_{k+n}$.

(b) For each $n\in\mathbb{N}$, the map
$$ 
w(n\mathcal{P}\cap\mathbb{Z}^s)\rightarrow\{y^{w(a)}z^n\mid a\in n\mathcal{P}\cap\mathbb{Z}^s\},\ w(a)\mapsto y^{w(a)}z^n,
$$
is bijective. 
Thus, $|w(n\mathcal{P}\cap\mathbb{Z}^{s})|=\dim_K(A^w(\mathcal{P})_n)$.

(c) The existence of $t^{c_1}z^{d_1},\ldots,t^{c_r}z^{d_r}\in A(\mathcal{P})$ such that $A(\mathcal{P})=K[\{t^{c_i}z^{d_i}\mid i=1,\ldots,r\}]$ follows from \cite[Theorem~9.3.6]{monalg-rev}. 
We now show the equality:
$$
A^w(\mathcal{P})=K[\{y^{w(c_i)}z^{d_i}\mid i=1,\ldots,r\}].
$$ 
\quad The inclusion ``$\supset$'' is clear because $c_i\in d_i\mathcal{P}\cap\mathbb{Z}^s$ and $y^{w(c_i)}z^{d_i}\in A^w(\mathcal{P})$ for all $i$. 
To show the inclusion ``$\subset$'' take $y^{w(a)}z^n\in A^w(\mathcal{P})$. 
Then, $t^az^n\in A(\mathcal{P})$ and $a\in n\mathcal{P}\cap\mathbb{Z}^s$, and consequently  
\begin{align*}
(a,n)&=n_1(c_1,d_1)+\cdots+n_r(c_r,d_r),\ n_i\in \mathbb{N},\\
w(a)&= n_1w(c_1)+\cdots+n_r w(c_r),\ n=n_1d_1+\cdots+n_rd_r\quad \therefore\\ 
y^{w(a)}z^n&=(y^{w(c_1)}z^{d_1})^{n_1}\cdots (y^{w(c_r)}z^{d_r})^{n_r},
\end{align*}
and $y^{w(a)}z^n\in K[\{y^{w(c_i)}z^{d_i}\mid i=1,\ldots,r\}]$.

(d) By part (c), it suffices to show that $y^{w(c_j)}z^{d_j}$ is integral over $K[\{y^{w(v_i)}z\mid i=1,\ldots,m\}]$ for all $j$. 
Since $(c_j,d_j)\in\mathbb{R}_+\{(v_i,1)\}_{i=1}^m$, there is $p\in\mathbb{N}_+$ such that
\begin{align*}
p(c_j,d_j)&=n_{1,j}(v_1,1)+\cdots+n_{m,j}(v_m,1),\ n_{i,j}\in \mathbb{N},\\
pc_j&=n_{1,j}v_1+\cdots+n_{m,j}v_m,\\
pw(c_j)&=w(pc_j)= n_{1,j}w(v_1)+\cdots+n_{m,j}w(v_m),\ pd_j=n_{1,j}+\cdots+n_{m,j}\quad \therefore\\ 
(y^{w(c_j)}z^{d_j})^p&=y^{pw(c_j)}z^{pd_j}=\big(y^{w(v_1)}\big)^{n_{1,j}}\cdots\big(y^{w(v_m)} \big)^{n_{m,j}}z^{n_{1,j}+\cdots+n_{m,j}}\\
&=\big(y^{w(v_1)}z\big)^{n_{1,j}}\cdots\big(y^{w(v_m)}z \big)^{n_{m,j}},
\end{align*}
and $y^{w(c_j)}z^{d_j}$ is integral over $K[\{y^{w(v_i)}z\mid i=1,\ldots,m\}]$.

(e) First we show $\overline{L}\subset\overline{A^w(\mathcal{P})}\subset A(w(\mathcal{P}))$. 
By part (c) and the fact that the Ehrhart ring $A(w(\mathcal{P}))$ is normal \cite[Theorem~9.3.6(c)]{monalg-rev}, we need only show the inclusion $A^w(\mathcal{P})\subset A(w(\mathcal{P}))$. 
By part (c), it suffices to show that $y^{w(c_j)}z^{d_j}\in A(w(\mathcal{P}))$ for all $j$. 
Since $(c_j,d_j)\in\mathbb{R}_+\{(v_i,1)\}_{i=1}^m$, we obtain 
\begin{align*}
(c_j,d_j)&=\lambda_{1,j}(v_1,1)+\cdots+\lambda_{m,j}(v_m,1),\ \lambda_{i,j}\in \mathbb{R}_+,\\
w(c_j)&=\lambda_{1,j}w(v_1)+\cdots+\lambda_{m,j}w(v_m),\ d_j=\lambda_{1,j}+\cdots+\lambda_{m,j}\quad \therefore\\ 
(w(c_j),d_j)&=\lambda_{1,j}(w(v_1),1)+\cdots+\lambda_{m,j}(w(v_m),1).
\end{align*}
\quad Thus, $y^{w(c_j)}z^{d_j}\in A(w(\mathcal{P}))$ because $w(\mathcal{P})={\rm conv}(w(v_1),\ldots,w(v_m))\subset\mathbb{R}^p$.
The second assertion follows directly from \cite[Theorem~3.9]{ehrhart}.

(f) By parts (a)-(e), $A^w(\mathcal{P})$ is a finitely generated graded module over the standard graded algebra $K[\{y^{w(v_i)}z\mid i=1,\ldots,m\}]$.
Then, by the Hilbert-Serre theorem \cite[Theorem~5.1.4]{monalg-rev}, we get that the Hilbert series of $A^w(\mathcal{P})$ has the form $h(x)/(1-x)^{\dim(A^w(\mathcal{P}))}$, where $h(1)\neq 0$ and $h(x)\in\mathbb{Z}[x]$, and $H_{A^w(\mathcal{P})}(n)$ is a polynomial function of degree $\dim(A^w(\mathcal{P}))-1$. 
\end{proof} 


\section{Examples of weighted Ehrhart functions and their generating functions}\label{examples-section}

In this part, we provide examples using the procedures for
\textit{Macaulay}$2$ \cite{mac2} and \textit{Normaliz}
\cite{normaliz2} of Appendix~\ref{procedures-wehrhart}. Our examples
help to have a better grasp of weighted Ehrhart theory, especially
 for those who are new to this theory, and complement some of the
results on the subject.  

\begin{example}{\rm\cite[Example~4.9]{ehrhartw}}\label{example0}\rm\ Let $\mathcal{P}={\rm
conv}(v_1,v_2,v_3,v_4)$ be the unit square in $\mathbb{R}^2$, where
$v_1=(0,0)$, $v_2=(1,0)$, $v_3=(0,1)$, $v_4=(1,1)$, and let $w_i$,
$i=1,2,3$, be the linear functions  $w_1=t_1+t_2$, $w_2=2t_1+3t_2$ and $w_3=t_1$.    
The following list of weighted Ehrhart functions was used in
\cite{ehrhartw} to illustrate a novel approach relating 
classical Ehrhart rings and weighted Ehrhart functions
\cite[Theorem~4.6]{ehrhartw}. 
Using Procedure~\ref{procedure0}, we readily obtain: 	
\begin{align*}
E_\mathcal{P}(n)&=(n+1)^2,\\
E_\mathcal{P}^{w_1}(n)&=(n^3+2n^2+n),\\
E_\mathcal{P}^{w_2}(n)&=(5/2)n^3+5n^2+(5/2)n,\\
E_\mathcal{P}^{w_3}(n)&=(1/2)n^3+n^2+(1/2)n,\\
E_\mathcal{P}^{w_1w_2}(n)&=(35/12)n^4+(20/3)n^3+(55/12)n^2+(5/6)n,\\
E_\mathcal{P}^{w_1w_3}(n)&=(17/12)n^4+(19/6)n^3+(25/12)n^2+(1/3)n,\\
E_\mathcal{P}^{w_2w_3}(n)&=(7/12)n^4+(4/3)n^3+(11/12)n^2+(1/6)n, \mbox{ and }\\
E_\mathcal{P}^{w_1w_2w_3}(n)&=(11/6)n^5+(29/6)n^4+(25/6)n^3+(7/6)n^2 \mbox{ for all }n\geq 0
\end{align*}
\end{example}

\begin{example}\label{example0-bis}\rm  
Let $\mathcal{P}={\rm conv}(v_1,v_2)$, $v_1=0$, $v_2=1$, be the unit interval in $\mathbb{R}$. 
Consider the linear functions $w_i=t_1$, $i=1,2,3$ and the polynomial
$w=(t_1+1)^3$. 
Then, $E_\mathcal{P}(n)=n+1$. 
Using Procedure~\ref{procedure0}, one obtains the following equalities
\vspace{-2mm}
$$
1^3+2^3+\cdots+(n+1)^3=E_\mathcal{P}^{w}(n)= ((n+1)^2(n+2)^2)/4\mbox{ for all }n\geq 0.
\vspace{-5mm}
$$

\quad Note that $w$ is not a homogeneous polynomial, the
leading coefficient of $E_\mathcal{P}^{w}$ is $1/4$, and
$\int_{\mathcal{P}}w=15/4$ is not equal to $1/4$; cf.
Theorem~\ref{nonneg-poly}(d) and the comments after this.    
\end{example}

\begin{example}\label{example0-bis-bis}\rm  
Let $\mathcal{P}={\rm conv}(v_1,v_2,v_3)$, $v_1=(1,0)$, $v_2=(0,2)$, $v_3=(2,3)$, and let $w$ be the linear function $w=(2/5)t_1-(6/25)t_2$. 
Using Procedure~\ref{procedure0}, we obtain (cf.~Theorem~\ref{nonneg-poly}(d)):
\begin{align*}
E_\mathcal{P}(n)&=\frac{5}{2}n^2+\frac{3}{2}n+1,&F_\mathcal{P}(x)&=\frac{1+2x+2x^2}{(1-x)^3};\\
E_\mathcal{P}^{t_1}(n)&=\frac{5}{2}n^3+\frac{3}{2}n^2+n,&F_\mathcal{P}^{t_1}(x)&=\frac{x(2x^2+8x+5)}{(1-x)^4};\\
E_\mathcal{P}^{t_2}(n)&=\frac{25}{6}n^3+\frac{5}{2}n^2+\frac{4}{3}n,&F_\mathcal{P}^{t_2}(x)&=\frac{x(x+4)(3x+2)}{(1-x)^4};\\
E_\mathcal{P}^{w}(n)&=\frac{2}{25}n,&F_\mathcal{P}^{w}(x)&=\frac{(2/25)x}{(1-x)^2}.
\end{align*}
\end{example}

\begin{example}\label{example3}\rm  
Let $\mathcal{P}=C_2=[0,1]^2$ be the unit square, and let $g_k(t_1,t_2)=t_1^kt_2^k$. 
Then, using Procedure~\ref{procedure0}, the weighted Ehrhart polynomials and series for $k=0,\ldots,4$ are given by
\begin{align*}
E_\mathcal{P}^{g_k}(n)=&n^2+2n+1,\quad k=0,\\
F_\mathcal{P}^{g_k}(x)=&\frac{1+x}{(1-x)^3},\quad k=0,\\
E_\mathcal{P}^{g_k}(n)=&(1/4)n^4+(1/2)n^3+(1/4)n^2,\quad k=1,\\
F_\mathcal{P}^{g_k}(x)=&\frac{x^3+4x^2+x}{(1-x)^5},\quad k=1,\\
E_\mathcal{P}^{g_k}(n)=&(1/9)n^6+(1/3)n^5+(13/36)n^4+(1/6)n^3+(1/36)n^2,\quad k=2,\\
F_\mathcal{P}^{g_k}(x)=&\frac{x^5+18x^4+42x^3+18x^2+x}{(1-x)^7},\quad k=2,\\
E_\mathcal{P}^{g_k}(n)=&(1/16)n^8+(1/4)n^7+(3/8)n^6+(1/4)n^5+(1/16)n^4,\quad k=3,\\
F_\mathcal{P}^{g_k}(x)=&\frac{x^7+72x^6+603x^5+1168x^4+603x^3+72x^2+x}{(1-x)^9},\quad k=3,\\
E_\mathcal{P}^{g_k}(n)=&(1/25)n^{10}+(1/5)n^9+(23/60)n^8+(1/3)n^7+(22/225)n^6\\
&-(1/30)n^5-(1/45)n^4+(1/900)n^2, \, k=4,\\ 
F_\mathcal{P}^{g_k}(x)=&\frac{x^9+278x^8+6480x^7+35402x^6+60830x^5+35402x^4+6480x^3+278x^2+x}{(1-x)^{11}},\, k=4.
\end{align*}
\end{example}

\begin{example}\label{example4}\rm 
Let $\mathcal{P}$ be the polytope ${\rm conv}((1,0),(0,1),(1,1))$ and let $t_i\colon\mathbb{R}^2\rightarrow\mathbb{R}$ be the linear function defined by the polynomial $t_i$ in $S=\mathbb{R}[t_1,t_2]$, $i=1,2$. 
Note that $\deg(E_\mathcal{P}^{t_i})=3$ for $i=1,2$.
Then, using Procedure~\ref{procedure0}, we obtain
\begin{align*}
E_\mathcal{P}(n)&=(1/2)n^2+(3/2)n+1,\, n\geq 0,\\
E_\mathcal{P}^{t_i}(n)&=(1/3)n^3+n^2+(2/3)n,\, n\geq 0,\, i=1,2,\\
E_\mathcal{P}^{t_1-t_2}(n)&=E_\mathcal{P}^{t_1}(n)-E_\mathcal{P}^{t_2}(n)=0,\, n\geq 0.
\end{align*}
\quad The values of $E_\mathcal{P}^{t_i}(n)$ at $n=1,2,3,4$ are $2,8,20,40$, respectively. 
The weighted Ehrhart series of $\mathcal{P}$ relative to $t_i$ is $F_\mathcal{P}^{t_i}(x)=2x/(1-x)^4$, $i=1,2$.
\end{example}

\begin{example}\label{example4-bis}\rm  
We give an illustration of the equality $E_{\mathcal{P}_w}=E_\mathcal{P}^{w+1}=E_\mathcal{P}^w+E_\mathcal{P}$ of
Proposition~\ref{linear-weighted-ehrhart}. 
Let $\mathcal{P}={\rm conv}(v_1,v_2,v_3,v_4)$, $v_1=(1,1,0)$, $v_2=(0,1,1)$, $v_3=(1,0,1)$, $v_4=(1,1,7)$, and let $w$ be the linear function $w=t_1+t_2+t_3$. 
Then, $w(v_1)=2$, $w(v_2)=2$, $w(v_3)=2$, $w(v_4)=9$, and using Procedure~\ref{procedure11} we obtain
\begin{align*}
E_{\mathcal{P}_w}(n)&=(n+1)(n+2)(105n^2+115n+12)(1/24),\\
E_\mathcal{P}(n)&=(1/6)(n+1)(n+2)(7n+3),\\
E_\mathcal{P}^w(n)&=E_{\mathcal{P}_w}(n)-E_\mathcal{P}(n)=(1/8)n(n+1)(n+2)(35n+29),\\
F_\mathcal{P}^w(x)&=F_{\mathcal{P}_w}(x)-F_\mathcal{P}(x)=\frac{1+53x+51x^2}{(1-x)^5}-\frac{1+6x}{(1-x)^4}=\frac{3(16+19x)x }{(1-x)^5},
\end{align*}
where ${\mathcal{P}_w}={\rm conv}(\{(v_i,0),(v_i,w(v_i))\}_{i=1}^4)$. 
\end{example}

\begin{example}\label{example11}\rm Let $\mathcal{P}$ be the polytope ${\rm conv}((2,0),(0,2))$ and let $w\colon\mathbb{R}^2\rightarrow\mathbb{R}$ be the linear function $(t_1,t_2)\mapsto t_1+t_2$. 
Using Proposition~\ref{linear-weighted-ehrhart} and Procedure~\ref{procedure11}, we get 
\begin{align*}
E_\mathcal{P}^w(n)&=E_{\mathcal{P}_w}(n)-E_\mathcal{P}(n)=(1+4n+4n^2)-(1+2n)=2n(1+2n),\\
F_\mathcal{P}^w(x)&=F_{\mathcal{P}_w}(x)-F_\mathcal{P}(x)=\frac{1+6x+x^2}{(1-x)^3}-\frac{1+x}{(1-x)^2}=\frac{6x+2x^2}{(1-x)^3}.
\end{align*}
\quad Let $\omega$ be the affine function $(x_1,x_2)\mapsto x_1+x_2-1$. 
Then, using Corollary~\ref{affine-weighted-ehrhart} and adapting Procedure~\ref{procedure11}, we get 
\begin{align*}
E_\mathcal{P}^\omega(n)&=E_{\mathcal{Q}_1}(n)-2E_\mathcal{P}(n)=(1+4n+4n^2)-2(1+2n)=4n^2-1, \text{ and}\\
F_\mathcal{P}^\omega(x)&=F_{\mathcal{Q}_1}(x)-2F_\mathcal{P}(x)=\frac{1+6x+x^2}{(1-x)^3}-\frac{2(1+x)}{(1-x)^2}=\frac{-1+6x+3x^2}{(1-x)^3}.
\end{align*}
\end{example}

\begin{example}\label{example22}\rm Let $\mathcal{P}$ be the polytope ${\rm conv}(1,2)=[1,2]$ and let $w\colon\mathbb{R}\rightarrow\mathbb{R}$ be the function $t_1\mapsto t_1^2$. 
Then, using the formula $\sum_{i=1}^ni^2=[n(n+1)(2n+1)]/6$, we get 
\begin{align*}
E_\mathcal{P}^w(n)&=\sum_{i=n}^{2n}i^2=\left(\sum_{i=1}^{2n}i^2\right)-\Biggl(\sum_{i=1}^{n-1}i^2\Biggl)=\left[\frac{n(14n^2+15n+1)}{6}\right].
\end{align*}
\quad Setting $\mathcal{Q}={\rm conv}((1,0),(2,0),(1,1),(2,4))$ and using Procedure~\ref{procedure11}, we obtain 
\begin{align*}
E_\mathcal{Q}(n)-E_\mathcal{P}(n)& =\left(\frac{5}{2}n^2+\frac{7}{2}n+1\right)-(n+1)=\frac{5}{2}(n)(n-1).
\end{align*}
\quad Thus, ${\rm vol}(\mathcal{Q})=5/2$, the area under the curve $w(t_1)=t_1^2$ is $\int_{1}^{2}w(t_1)\mathrm{d}t_1=7/3$, $E_\mathcal{P}^w(n)$ is a polynomial in $n$ of degree $3$ whereas $E_\mathcal{Q}(n)-E_\mathcal{P}(n)$ is a polynomial in $n$ of degree $2$. 
This example shows that the linearity assumption in Proposition~\ref{linear-weighted-ehrhart} is essential. 
\end{example}

\begin{example}\label{example6}\rm Let $\mathcal{P}$ be the lattice polytope ${\rm conv}(v_1,\ldots,v_7)$ defined by the points
\begin{align*}
&v_1=(1, 1, 0, 0, 0, 0, 0), v_2=(0, 1, 1, 0, 0, 0, 0), v_3=(0, 0, 1, 1, 0, 0,0), v_4=(1, 0, 0, 1, 0, 0, 0),\\
&v_5= (0, 0, 0, 0, 1, 1, 0), v_6=(0, 0, 0, 0, 0, 1, 1), v_7=(0, 0, 0, 0, 1, 0, 1),
\end{align*}
and let $w=t_1\cdots t_s$. 
Then, using Procedure~\ref{procedure0}, we obtain
\begin{align*}
E_\mathcal{P}(n)&=(1/120)(n+1)(n+2)(n+3)(n+4)(2n+5),\\
E_\mathcal{P}^w(n)&=(1/59875200)(n)(n-3)(n-2)(n-1)(n+1)(n+2)(n+3)(n+4)\\
&\ \ \ \ (35n^4+40n^3-143n^2+122n-810)\mbox{ for all }n\geq 0,
\end{align*}
$E_\mathcal{P}^w(n)=0$ for $n=0,1,2,3$, and $E_\mathcal{P}^w(4)=6$.
\end{example}


\begin{example}\label{jun4-25-1}\rm 
Let $\mathcal{P}={\rm conv}((1,0),\,(0,1),\,(1,1))$ and let
$f=2x+3y$. To compute $\int_\mathcal{P}f$ and the weighted Ehrhart
function of $\mathcal{P}$, using \textit{Normaliz}
\cite{normaliz2}, one needs to find the inequalities that define
$\mathcal{P}$ first. Below we set up a procedure for \textit{Normaliz} 
to compute $\int_\mathcal{P}f$. On the other hand using Procedure~\ref{procedure0}
requires the points that define $\mathcal{P}$ instead. Note that one has:
\begin{enumerate}
\item[\rm(a)] $\mathcal{P}=\{(x,y)\mid 0\leq x\leq 1,\, 1-x\leq y\leq 1\}$, and  
\item[\rm(b)] $\int_{0}^1\int_{1-x}^1(ax+by)\,dy\, dx=(a+b)/3$.
\end{enumerate}
\quad Using the inequalities for $\mathcal{P}$ of part
(a) and the following input file for \textit{Normaliz}
\cite{normaliz2} gives the equality $\int_\mathcal{P}f=5/3$. 
By a similar procedure, we get $\int_\mathcal{P}(x^2+y^2)=1/2$.
\begin{verbatim}
amb_space 3
inequalities 4
1 0 0
1 1 1
-1 0 -1
0 -1 -1
polynomial
2*x[1]+3*x[2];
Integral
\end{verbatim}
\quad If $g=x^2+y^2$, using the following input file for \textit{Normaliz} \cite{normaliz2}, we obtain that the weighted Ehrhart function of $\mathcal{P}$ is $E_\mathcal{P}^g(n)=(1/3)n+(3/2)n^2+(5/3)n^3+(1/2)n^4$ for $n\geq 0$ and the weighted Ehrhart series of $\mathcal{P}$ is $F_\mathcal{P}^g(x)=(4x+8x^2)/(1-x)^5$.   
\vspace{-2mm}
\begin{verbatim} 
amb_space 3
inequalities 4
1 0 0
1 1 1
-1 0 -1
0 -1 -1
polynomial
x[1]^2+x[2]^2;
WeightedEhrhartSeries
\end{verbatim}
\end{example}

\begin{example}\label{example7} Let $w_1$ be the linear function $t_1+2t_2$ and let $\mathcal{P}$ be the polytope ${\rm conv}(v_1,v_2,v_3)$, where $v_1=(1,1)$, $v_2=(3,0)$, $v_3=(2,3)$.
Letting $v_4=(2,2)$, $v_5=(2,1)$, one has
\begin{align*}
w_1(\mathcal{P}\cap\mathbb{Z}^2)&=w_1(\{v_1,v_2,v_3,v_4,v_5\})=\{3,3,8,6,4\}=\{3,4,6,8\},\\
w_1(\mathcal{P})\cap\mathbb{Z}&={\rm conv}(\{3,3,8\})\cap\mathbb{Z}=[3,8]\cap\mathbb{Z}=\{3,4,5,6,7,8\},
\end{align*}
$|nw_1(\mathcal{P})\cap\mathbb{Z}|=|[3n,8n]\cap\mathbb{Z}|=5n+1$ for $n\geq 0$, and $w_1(\mathcal{P}\cap\mathbb{Z}^2)\subsetneq w_1(\mathcal{P})\cap\mathbb{Z}$. 
The Ehrhart ring of $\mathcal{P}$ is given by
$$
A(\mathcal{P})=K[t_1t_2z,\ t_1^3z,\ t_1^2t_2^3z,\ t_1^2t_2^2z,\ t_1^2t_2z],
$$
$E_\mathcal{P}(n)=1+(3/2)n+(5/2)n^2$, and $F_\mathcal{P}(x)=(1+2x+2x^2)/(1-x)^3$. 
Then, by Theorem~\ref{jdl2}, we obtain
$$
A^w(\mathcal{P})=K[\{y_1^{w_1(a)}z^n\mid a\in
n\mathcal{P}\cap\mathbb{Z}^2\}]
=K[y_1^3z,\ y_1^3z,\ y_1^8z,\ y_1^6z,\ y_1^4z],
$$
where monomials in $A^w(\mathcal{P})$ of the form $y_1^mz^n$ have degree $n$. 
The Hilbert function $H(n)$ of $A^w(\mathcal{P})$ is the number of distinct monomials in $w_1(n\mathcal{P}\cap\mathbb{Z}^2)$. 
The toric ideal $P$ of $A^w(\mathcal{P})$ is the kernel of the map 
$$K[u_1,\ldots,u_5]\mapsto A^w(\mathcal{P}),\  u_i\mapsto
y_1^{n_i}z,
$$
where $n_1=3$, $n_2=3$, $n_3=8$, $n_4=6$, $n_5=4$. Then
$$
P=(u_1-u_2,\  u_4^2-u_3u_5,\ u_2^2u_4-u_5^3,\ u_2^2u_3-u_4u_5^2),
$$
$H(n)=5n-1$ for $n\geq 1$, and 
$$F_{A^w(\mathcal{P})}(x)=\frac{1+2x+2x^2}{(1-x)^2}.$$ 
\quad The ring $A^w(\mathcal{P})$ is Cohen--Macaulay because the
projective dimension of $K[u_1,\ldots,u_5]/P$ is equal to $3$ and is equal to
the height of the toric ideal $P$. 
The ring $A^w(\mathcal{P})$ is not normal because its integral closure
is given by $\overline{A^w(\mathcal{P})}=K[\{y_1^iz\mid i=3,\ldots,8\}]$.
\end{example}

\begin{appendix}

\section{Procedures for weighted Ehrhart functions}\label{procedures-wehrhart}

In this section, we give procedures for \textit{Macaulay}$2$
\cite{mac2}---using its interface to \textit{Normaliz}
\cite{normaliz2}---to determine the weighted Ehrhart polynomial
$E_\mathcal{P}^w$ and the weighted Ehrhart series $F_\mathcal{P}^w$
of a lattice polytope $\mathcal{P}$.     

\begin{procedure}\label{procedure0}\rm 
Let $\mathcal{P}={\rm conv}(v_1,\ldots,v_m)$ be a lattice polytope in
$\mathbb{R}^s$ and let
$\omega\colon\mathbb{R}^s\rightarrow\mathbb{R}$ be a polynomial.   
The following procedure for \textit{Macaulay}$2$ \cite{mac2} computes
$E_\mathcal{P}^w$ and its generating function
$F_\mathcal{P}^w(x)=\sum_{n=0}^\infty E_\mathcal{P}^w(n)x^n$ using
Theorem~\ref{nonneg-poly}(a), Lemma~\ref{oct12-22}, and Lagrange
interpolation \cite{interpolation}.     
One can also use finite differences methods to compute
$E_\mathcal{P}^w$ \cite{Sauer,sauer,Sauer-Xu}, and \textit{Normaliz}
\cite{normaliz2} to compute $E_\mathcal{P}^w$ and $F_\mathcal{P}^w$
\cite{Bruns-Soger}.    
The following procedure corresponds to Example~\ref{example3}.  
We load the \textit{Macaulay}$2$ packages \textit{Polyhedra}
\cite{Polyhedra} and \textit{Normaliz} \cite{normaliz2} for
computations with polyhedra.  
\begin{verbatim}
restart
loadPackage("Normaliz",Reload=>true)
loadPackage("Polyhedra", Reload => true)
R=QQ[t1]
--Euler numbers
H1=(d,k) -> sum for i from 0 to k list (-1)^i*binomial(d+1,i)*(k-i)^d;
--This is the Ehrhart series of the d-th cube [0,1]^d, for d>=1
H=(d) -> (sum for k from 1 to d list H1(d,k)*t1^(k-1))/(1-t1)^(d+1)
Fc=(d)-> if d==0 then 1/(1-t1) else t1*H(d) 
--This function gives the Ehrhart series of a polynomial f
--to be used later after we compute the 
--weighted Ehrhart polynomial using 
--Lagrange interpolation 
HS=(f)->(entries(matrix{reverse apply(terms f,x->leadCoefficient x)}*
(vector toList apply(flatten apply(reverse terms f, 
x->degree(x)),Fc))))#0
--This is the cube P=[0,1]^2
X=toList (set{0,1})^**(2)/deepSplice
A=matrix apply(X,x->toList(x))
--For technical reasons we transpose A
P = convexHull transpose A 
--Ehrhart polynomial of P
E=ehrhart P
Ee=toExternalString E
g=value Ee
--dimension of the polytope P
d=(degree E)#0
k=1
--This is the polynomial w that we use as a weight
w=(t1,t2)->t1^k*t2^k
--degree of w
p=2*k
--d+p is the degree of E_P^w
--the variable t3 below is needed because 
--the integral points of the dilation nP come with a 1 
--at the end. 
--number of variables=ambient space of P plus 1. 
w1=(t1,t2,t3)->t1^k*t2^k+t3-1
--integral points of the dilation nP
nP=(n)->delete({0},(apply(entries (normaliz(n*A,"polytope"))
#"gen",x-> if last x == 1 then x else {0})))
--weighted Ehrhart polynomial in degree n
Ew=(n)->sum apply(apply(nP(n),x->toSequence(x)),w1)
--interpolation points
puntos=toList apply(1..d+p+1,Ew)
--This is the Macaulay2 program of D. Cook II 
--and C. Jansen for Lagrange polynomial interpolation
polynomialInterpolation = method ()
polynomialInterpolation (List,PolynomialRing) := 
RingElement => (L,R) -> (if #gens R != 1 then error 
"must be a single variable polynomial ring";
n := #L;
t := first gens R;
p := for j from 0 to n-1 list (
L_j * product(0..(j-1) | (j+1) .. (n-1), 
k -> (t - (k+1)) / ((j+1) - (k+1)))
);
sum(p)
)
polynomialInterpolation (List) := RingElement => L -> (
t := symbol t;
polynomialInterpolation(L, QQ[t])
)					    
--This is weighted Ehrhart polynomial of P obtained 
--by Lagrange interpolation.
f=polynomialInterpolation(puntos, R)
factor f
--This is the weighted Ehrhart series of P
HS(f)
factor(HS(f))
--this gives the invariants of P 
--and the usual Ehrhart series of P
ehrhartseries=(normaliz(A,"polytope"))#"inv"
ehrhartseries#"hilbert series num"
ehrhartseries#"hilbert series denom"
\end{verbatim}
\end{procedure}

\begin{procedure}\label{procedure11}\rm 
Let $\mathcal{P}={\rm conv}(v_1,\ldots,v_m)$ be a lattice polytope in
$\mathbb{R}^s$ and let
$\omega\colon\mathbb{R}^s\rightarrow\mathbb{R}$ be a linear function
such that $w(v_i)\in\mathbb{N}_+$ for $i=1,\ldots,m$ and
$w(e_i)\in\mathbb{N}$ for $i=1,\ldots,s$.     
The following procedure for \textit{Normaliz} \cite{normaliz2}
computes the weighted Ehrhart function $E_\mathcal{P}^w$ and its generating function
$F_\mathcal{P}^w(z)=\sum_{n=0}^\infty E_\mathcal{P}^w(n)z^n$ using
the formula   
$$
E_\mathcal{P}^w(n):=E_\mathcal{P}^{w}(n)=\sum_{a\in n\mathcal{P}\cap\mathbb{Z}^s}w(a)=
E_\mathcal{Q}(n)-E_\mathcal{P}(n)\mbox{ for }n\geq 0,
$$
where $\mathcal{Q}={\rm
conv}((v_1,0),\ldots,(v_m,0),(v_1,w(v_1)),\ldots(v_m,w(v_m)))$ (see
Proposition~\ref{linear-weighted-ehrhart}).   
The input to compute $E_\mathcal{Q}$ (resp. $E_\mathcal{P}$) is the
matrix whose rows are the vectors defining $\mathcal{Q}$ (resp.
$\mathcal{P}$).    
This procedure corresponds to Example~\ref{example11}. 
\begin{verbatim}
amb_space 4
polytope 4
2 0 0
0 2 0
2 0 2
0 2 2
/*
This computes E_Q and F_Q. 
To compute E_P and F_P use the following:
amb_space 3
polytope 2
2 0 
0 2 
*/
\end{verbatim}
\end{procedure}
\end{appendix}

\section*{Acknowledgments.} 
\textit{Macaulay}$2$ \cite{mac2} and \textit{Normaliz}
\cite{normaliz2} were used to implement algorithms for computing
weighted Ehrhart functions and series of lattice polytopes.   

\section*{Statements and Declarations}  
On behalf of all authors, the corresponding author states that there is no conflict of interest.

No funding was received for conducting this study.

The authors have no relevant financial or non-financial interests to disclose.

Data sharing is not applicable to this article as no datasets were generated or analyzed during the current study. 


\bibliographystyle{plain}

\end{document}